\documentclass[12pt,oneside,english]{amsart}
\usepackage[T1]{fontenc}
\usepackage[latin9]{inputenc}
\usepackage{geometry}
\usepackage{amstext}
\usepackage{amsthm}
\usepackage{amssymb}

\makeatletter
\numberwithin{equation}{section}
\numberwithin{figure}{section}
\theoremstyle{plain}
\newtheorem*{thm*}{\protect\theoremname}
\theoremstyle{plain}
\newtheorem{thm}{\protect\theoremname}
\theoremstyle{plain}
\newtheorem{lem}[thm]{\protect\lemmaname}

\DeclareMathOperator{\supp}{supp}
\DeclareMathOperator{\dist}{dist}
\usepackage{hyperref}

\makeatother

\usepackage{babel}
\providecommand{\lemmaname}{Lemma}
\providecommand{\theoremname}{Theorem}

\begin{document}
\title{One-sided $C_{p}$ estimates via $M^{\sharp}$ function}
\author{María Lorente}
\email{m\_lorente@uma.es }
\address{Departamento de Análisis Matemático, Estadística e Investigación Operativa
y Matemática Aplicada. Facultad de Ciencias. Universidad de Málaga
(Málaga, Spain)}
\author{Francisco J. Martín-Reyes}
\email{martin\_reyes@uma.es}
\address{Departamento de Análisis Matemático, Estadística e Investigación Operativa
y Matemática Aplicada. Facultad de Ciencias. Universidad de Málaga
(Málaga, Spain)}
\author{Israel P. Rivera-Ríos}
\email{israelpriverarios@uma.es}
\address{Departamento de Análisis Matemático, Estadística e Investigación Operativa
y Matemática Aplicada. Facultad de Ciencias. Universidad de Málaga
(Málaga, Spain). Departamento de Matemática. Universidad Nacional
del Sur (Bahía Blanca, Argentina).}
\begin{abstract}
We recall that $w\in C_{p}^{+}$ if there exist $\varepsilon>0$ and
$C>0$ such that for any $a<b<c$ with $c-b<b-a$ and any measurable
set $E\subset(a,b)$, the following holds
\[
\int_{E}w\leq C\left(\frac{|E|}{(c-b)}\right)^{\varepsilon}\int_{\mathbb{R}}\left(M^{+}\chi_{(a,c)}\right)^{p}w<\infty.
\]
This condition was introduced by Riveros and de la Torre \cite{RdT}
as a one-sided counterpart of the $C_{p}$ condition studied first
by Muckenhoupt and Sawyer \cite{M,S}. In this paper we show that
given $1<p<q<\infty$ if $w\in C_{q}^{+}$ then
\[
\|M^{+}f\|_{L^{p}(w)}\lesssim\|M^{\sharp,+}f\|_{L^{p}(w)}
\]
and conversely if such an inequality holds, then $w\in C_{p}^{+}.$

This result is the one-sided counterpart of Yabuta's main result in
\cite{Y}. Combining this estimate with known pointwise estimates
for $M^{\sharp,+}$ in the literature we recover and extend the result
for maximal one-sided singular integrals due to Riveros and de la
Torre \cite{RdT} obtaining counterparts a number of operators. 
\end{abstract}

\maketitle

\section{Introduction and Main Result}

One-sided theory of weights was begun by Sawyer in \cite{Smax} where
he provided the characterization of the two weighted inequalities
for the one-sided maximal functions 
\[
M^{+}f(x)=\sup_{h>0}\frac{1}{h}\int_{x}^{x+h}|f(y)|dy\qquad M^{-}f(x)=\sup_{h>0}\frac{1}{h}\int_{x-h}^{x}|f(y)|dy.
\]
In some sense it is a somehow curious fact that the results in \cite{Smax}
appeared more than a decade later than the characterization of the
one weight inequalities for the maximal function due to Muckenhoupt
\cite{Mmax} if one bears in mind that actually the maximal operators
studied by Hardy and Littlewood \cite{HL} were actually $M^{+}$
and $M^{-}$. 

Since Sawyer's work a number of papers such as \cite{AFMR,FMRO,GMROPSdT,LMR,LMR-1,MR,MRO,MROdT,MRdT2,MRPdT,O,OdR}
and even more that we will cite througout this paper were devoted
to develop the one-sided theory. However, at this point, we believe
it is worth mentioning some papers which have expanded the field of
one-sided estimates in the last years. Kinnunen and Saari \cite{KS,KS1}
studied parabolic Muckenhoupt conditions in connection with PDEs and
more recently Hytönen and Rosén devoted their work \cite{HR} to causal
sparse domination motivated by about maximal regularity estimates
for elliptic PDEs, obtaining results related to one-sided weighted
estimates for singular integrals.

A well known estimate in theory of weights that was settled by Coifman
and Fefferman, says that if $w\in A_{\infty}$ then, for every $p\in(0,\infty)$
we have that
\begin{equation}
\|Tf\|_{L^{p}(w)}\leq c_{p,n,w}\|Mf\|_{L^{p}(w)}\label{eq:CF}
\end{equation}
where $T$ stands for any Calderón-Zygmund operator.

Although $w\in A_{\infty}$ is sufficient for (\ref{eq:CF}) to hold,
it turns out to not to be necessary. Muckenhoupt \cite{M}, showed
that if (\ref{eq:CF}) holds for the Hilbert transform for $p>1$
then there exist $c,\varepsilon>0$ such that for every cube $Q$
and every measurable subset $E\subset Q$, 
\[
w(E)\leq c\left(\frac{|E|}{|Q|}\right)^{\varepsilon}\int_{\mathbb{R}}\left(M\chi_{Q}\right)^{p}w,
\]
namely $w\in C_{p}$. Later on, Sawyer \cite{S} showed that if $1<p<q<\infty$
and $w\in C_{q}$, then (\ref{eq:CF}) holds. 

Yabuta \cite{Y} provided a different approach to the question. He
showed that if $1<p<q<\infty$ and $w\in C_{q}$ 
\begin{equation}
\|Mf\|_{L^{p}(w)}\leq c_{p,n,w}\|M^{\sharp}f\|_{L^{p}(w)}\label{eq:Y}
\end{equation}
and also that if such an inequality holds then $w\in C_{p}$. An alternative
proof of this estimate and a slight generalization of the $C_{p}$
condition was studied by Lerner in \cite{L}. 

In the last years some advances have been made in the study of this
kind of questions. Lerner \cite{L2} fully characterized the weak
type version of (\ref{eq:CF}). Sawyer's result has been extended
to the full range in \cite{CLPRR} and also quantitative estimates
in terms of a suitable $C_{p}$ constant and further operators, such
as rough singular integrals, have been explored in \cite{C,CLTR}.

In the one-sided setting we are aware of just one work in this direction
in which Riveros and de la Torre \cite{RdT} introduced the one-sided
version of the $C_{p}$ condition, which reads as follows. We say
that $w\in C_{p}^{+}$ if there exist $\varepsilon>0$ and $C>0$
such that for any $a<b<c$ with $c-b<b-a$ and any measurable set
$E\subset(a,b)$, the following holds
\begin{equation}
\int_{E}w\leq C\left(\frac{|E|}{c-b}\right)^{\varepsilon}\int_{\mathbb{R}}\left(M^{+}\chi_{(a,c)}\right)^{p}w.\label{eq:Cp+}
\end{equation}
The main result in that work was the following one-sided counterpart
of \cite{S}.
\begin{thm*}[{\cite[Theorem 1]{RdT}}]
Let $1<p<q<\infty$. If $w\in C_{q}^{+}$ and $(T^{+})^{*}$ is a
maximal Calderón-Zygmund one-sided singular integral, then 
\begin{equation}
\int_{\mathbb{R}}|(T^{+})^{*}f|^{p}w\leq C\int_{\mathbb{R}}(M^{+}f)^{p}w.\label{eq:OneSidedSawyer}
\end{equation}
\end{thm*}
Observe that in \cite{RdT}, additionally the authors assume that
the integral in the right hand side of (\ref{eq:Cp+}) is finite.
Note that if the right hand side of the condition $C_{q}^{+}$ is
not finite, then the same happens to $\int_{\mathbb{R}}(M^{+}f)^{p}w$
and hence the inequality is trivial. 

Note that the $C_{p}^{+}$ class is defined in terms of the one-sided
maximal operator $M^{+}$. We will review the definitions of $M^{+}$
and the remainder of the one-sided operators studied in this paper
in section \ref{sec:Preliminaries-and-definitions}. We would like
to observe as well that the assumption $c-b<b-a$ can be dropped.
Assume that $b-a\leq c-b$. Let $\bar{a}<a$ such that $a-\bar{a}=c-a$.
Note, that then, $b\in(\bar{a},c)$, and $c-b<b-\bar{a}.$ On the
other hand observe that $M^{+}\chi_{(\bar{a},c)}\simeq M^{+}\chi_{(a,c)}$
and hence the $C_{p}^{+}$ condition would hold just with a larger
constant $C$ but without the restriction $c-b<b-a$. 

The purpose of this paper is to provide a one-sided counterpart of
Yabuta's characterization (\ref{eq:Y}) and to derive a number of
new results relying upon it. The precise statement of our theorem
is the following.
\begin{thm}
\label{thm:OneSidedYabuta}Let $1<p<q<\infty$. If $f\in L_{c}^{\infty}$
and $w\in C_{q}^{+}$ then

\[
\|M^{+}f\|_{L^{p}(w)}\lesssim\|M^{\sharp,+}f\|_{L^{p}(w)},
\]
provided the left-hand side of the estimate is finite. Conversely
if the preceding estimate holds, then $w\in C_{p}^{+}.$
\end{thm}

We would like to note that the corresponding counterpart for $T^{-}$
operators holds as well. However, here and throughout the remainder
of this paper we will just deal with the case of $T^{+}$ operators. 

Exploiting the approach in \cite{CLPRR} we shall derive a number
of consequences of this result. Among them we will recover the result
for one-sided Calderón-Zygmund singular integrals due to de la Torre
and Riveros that we stated above. We present those results in section
\ref{sec:Applications}.

The remainder of the paper is organized as follows. We devote Section
\ref{sec:Preliminaries-and-definitions} to gather some results and
definitions that will be useful throughout the remainder of the work.
In Section \ref{sec:Applications} we present the applications of
Theorem \ref{thm:OneSidedYabuta}, namely counterparts of (\ref{eq:OneSidedSawyer})
for some other one-sided operators, and even for one-sided Calderón-Zygmund
singular integrals themselves. Section \ref{sec:Proof-MT} is devoted
to the proof of Theorem \ref{thm:OneSidedYabuta}. Additionally we
provide an appendix settling a suitable Cotlar inequality that we
have not been able to find in the literature and that will be useful
for us to recover and generalize \cite[Theorem 1]{RdT}.

\section{\label{sec:Preliminaries-and-definitions}Preliminaries and definitions}

We recall that the one-sided maximal function $M^{+}$ is defined,
as we noted in the introduction, as
\[
M^{+}f(x)=\sup_{h>0}\frac{1}{h}\int_{x}^{x+h}|f|
\]
and the sharp maximal function $M^{\sharp,+}$, that was introduced
in \cite{MRdT}, as
\[
M^{\sharp,+}f(x)=\sup_{h>0}\frac{1}{h}\int_{x}^{x+h}\left(f(y)-\frac{1}{h}\int_{x+h}^{x+2h}f\right)^{+}dy.
\]

Another class of operators that we will be dealing with and that have
already appeared in the previous section are one-sided singular integral
operators that were introduced in \cite{AFMR}. We say that a function
$K\in L_{\text{loc}}^{1}(\mathbb{R}\setminus\{0\})$ is a Calderón-Zygmund
kernel if the following properties hold.
\begin{enumerate}
\item There exists a finite constant $B_{1}$ such that 
\[
\left|\int_{\varepsilon<|x|<N}K(x)dx\right|\leq B_{1}
\]
for all $0<\varepsilon<N$. Furthermore, $\lim_{\varepsilon\rightarrow0^{+}}\int_{\varepsilon<|x|<N}K(x)dx$
exists.
\item There exists a constant $B_{2}$ such that
\[
|K(x)|\leq\frac{B_{2}}{|x|}
\]
for all $x\not=0$. 
\item There exists a finite constant $B_{3}$ such that
\begin{equation}
|K(x-y)-K(x)|\leq B_{3}\frac{|y|}{|x|^{2}}\label{eq:HLipschitz}
\end{equation}
for all $x$ and $y$ with $|x|>2|y|>0$.
\end{enumerate}
We say that $T^{+}$ is a one-sided Calderón-Zygmund singular integral
if 
\begin{equation}
T^{+}f(x)=\lim_{\varepsilon\rightarrow0}\int_{x+\varepsilon}^{\infty}K(x-y)f(y)dy\label{eq:defT+}
\end{equation}
where $K$ is a Calderón-Zygmund kernel with support in $\mathbb{R}^{-}$. 

We would like to emphasize that this kind of operators are Calderón-Zygmund
operators, and hence they have all the usual properties of operators
in that class, but with the extra feature that $K$ is supported in
$\mathbb{R}^{-}$. Examples of such operators are provided in \cite{AFMR}.

Replacing (\ref{eq:HLipschitz}) by some other smoothness conditions
we obtain some more operators. For instance, we may assume that there
exist numbers $c_{r},C_{r}>0$ such that for any $y\in\mathbb{R}$
and $R>c_{r}|y|$, 
\begin{equation}
\sum_{m=1}^{\infty}2^{m}R\left(\frac{1}{2^{m}R}\int_{2^{m}<|x|\leq2^{m+1}R}|K(x-y)-K(x)|^{r}dx\right)^{\frac{1}{r}}\leq C_{r}\label{eq:LrHormander}
\end{equation}
if $1\leq r<\infty$ and
\[
\sum_{m=1}^{\infty}2^{m}R\sup_{2^{m}<|x|\leq2^{m+1}R}|K(x-y)-K(x)|\leq C_{\infty}
\]
if $r=\infty$. If $K$ satisfies an $L^{r}$-Hörmander condition
we say that $K\in\mathcal{H}_{r}$. This yields that we may define
an operator $T^{+}$ exactly as we did in (\ref{eq:defT+}), but with
$K$ satisfying (\ref{eq:LrHormander}) instead of (\ref{eq:HLipschitz}).
We may go even further. Let us recall first the notion of Orlicz average.
Let $A:[0,\infty)\rightarrow[0,\infty)$ a Young function, namely
a convex function such that $A(0)=0$, $A(1)=1$ and $\lim_{t\rightarrow\infty}A(t)=\infty$.
Given a measurable set $E$ we define the average of $f$ over $E$
with respect to $A$ as 
\[
\|f\|_{A,E}=\inf\left\{ \lambda>0\,:\,\frac{1}{|E|}\int_{E}A\left(\frac{|f(x)|}{\lambda}\right)dx\leq1\right\} .
\]
Relying upon that definition we may define the maximal function $M_{A}^{+}$
as follows
\[
M_{A}^{+}f(x)=\sup_{h>0}\|f\|_{A,[x,x+h]}.
\]
It is also worth mentioning that we can define a function associated
to $A$, that we call $\overline{A}$, which turns out to be a Young
function as well and satisfies the following inequalities 
\[
t\leq\overline{A}^{-1}(t)A^{-1}(t)\leq2t.
\]
Furthermore it can be shown that if $A_{1},A_{2},\dots,A_{n}$ are
Young functions such that 
\[
A_{1}^{-1}(t)\dots A_{n}^{-1}(t)\lesssim t
\]
then
\[
\frac{1}{|E|}\int_{E}|f_{1}\dots f_{n}|\lesssim\|f_{1}\|_{A_{1},E}\dots\|f_{n}\|_{A_{n},E}.
\]
Coming back to the previous discussion, as we mentioned above, we
may define a class of kernels generalizing the $L^{r}$-Hörmander
condition. Given a Young function $A$ we say that it satisfies an
$L^{A}$-Hörmander condition if there exist $c_{A},C_{A}\geq1$ such
that for any $x\in\mathbb{R}$ and $R>c_{A}|x|$
\begin{equation}
\sum_{m=1}^{\infty}2^{m}R\|(K(x-\cdot)-K(-\cdot))\chi_{2^{m}R<|\cdot|\leq2^{m+1}R}\|_{A,B(0,2^{m+1}R)}\leq C_{A}.\label{eq:AHormander}
\end{equation}
If $K$ satisfies this condition we say that $K\in\mathcal{H}_{A}$.
In order to be able to deal with commutators we introduce another
condition. We say that $K\in\mathcal{H}_{A,k}$ if there exist $c_{A,k},C_{A,k}\geq1$
such that for any $x\in\mathbb{R}$ and $R>c_{A,k}|x|$
\[
\sum_{m=1}^{\infty}2^{m}Rm^{k}\|(K(x-\cdot)-K(-\cdot))\chi_{2^{m}R<|\cdot|\leq2^{m+1}R}\|_{A,B(0,2^{m+1}R)}\leq C_{A,k}.
\]

In both cases, whether $K\in\mathcal{H}_{A}$ or $K\in\mathcal{H}_{A,k}$,
we may define a singular integral operator exactly as we did in (\ref{eq:defT+}).
Those classes of kernels were introduced and studied in \cite{LRdT1,LMRdT}.

We would like to end recalling that we may define a maximal version
of any of the singular integral operators that we have just presented
in this section as follows 
\[
(T^{+})^{*}f(x)=\sup_{\varepsilon>0}\left|T_{\varepsilon}^{+}f(x)\right|=\sup_{\varepsilon>0}\left|\int_{\varepsilon+x}^{\infty}K(x-y)f(y)dy\right|.
\]

\section{\label{sec:Applications}Corollaries of the main theorem}

As we announced in the previous section we will derive a number of
applications of Theorem \ref{thm:OneSidedYabuta}. Our use of that
theorem will rely upon the following Lemma.
\begin{lem}
\label{Lem:YabutaDelta}Let $0<p<\infty$. If $\delta\in(0,p)$ then,
if $w\in C_{\rho}^{+}$ with $\rho>\frac{p}{\delta}$, we have that
\[
\|M_{\delta}^{+}f\|_{L^{p}(w)}\lesssim\|M_{\delta}^{\sharp,+}f\|_{L^{p}(w)}
\]
where $M_{\delta}^{+}(f)=\left(M^{+}(|f|^{\delta})\right)^{\frac{1}{\delta}}$
and $M_{\delta}^{\sharp,+}(f)=\left(M^{\sharp,+}(|f|^{\delta})\right)^{\frac{1}{\delta}}$.
\end{lem}

\begin{proof}
Observe that, since $\delta\in(0,p)$ we have that $\frac{p}{\delta}>1$.
Taking that into account and the fact that $w\in C_{\rho}^{+}$ we
have by Theorem \ref{thm:OneSidedYabuta} that 
\begin{align*}
\|M_{\delta}^{+}f\|_{L^{p}(w)}^{p} & =\int_{\mathbb{R}}M^{+}(|f|^{\delta})^{\frac{p}{\delta}}(x)w(x)dx\\
 & \lesssim\int_{\mathbb{R}}M^{\sharp,+}(|f|^{\delta})^{\frac{p}{\delta}}(x)w(x)dx\\
 & =\|M_{\delta}^{\sharp,+}f\|_{L^{p}(w)}^{p}
\end{align*}
and we are done.
\end{proof}

\subsection{Singular integral operators, $L^{A}$-Hörmander operators and their
commutators}

In this section we present our results for singular integral operators,
$A$-Hörmander operators and their commutators. We will provide some
full arguments here, that we shall omit for the remainder of the corollaries
since they will be analogous to the ones provided here.

We begin recalling that for one-sided Calderón-Zygmund singular integrals
$T^{+}$ it was shown in \cite[Lemma 1]{LR} that for $0<\delta<1$,
\begin{equation}
M_{\delta}^{\sharp,+}(T^{+}f)\lesssim M^{+}f.\label{eq:MsharpCZO}
\end{equation}
Using that estimate we can derive the following result. 
\begin{thm}
\label{thm:CZO}Let $0<p<\infty$ and $\varepsilon>0$ and assume
that $w\in C_{\max\{p,1\}+\varepsilon}^{+}.$ Then 

\[
\|T^{+}f\|_{L^{p}(w)}\lesssim\|M^{+}f\|_{L^{p}(w)}.
\]
\end{thm}

\begin{proof}
Let $\delta\in(0,1)$ such that $1<\frac{p}{\delta}<\max\{p,1\}+\varepsilon$.
For that choice of $\delta$, taking into account (\ref{eq:MsharpCZO})
and Lemma \ref{Lem:YabutaDelta} with $\rho=\max\{p,1\}+\varepsilon$,
we have that
\begin{align*}
\|T^{+}f\|_{L^{p}(w)} & \lesssim\|M_{\delta}^{+}(T^{+}f)\|_{L^{p}(w)}\lesssim\|M_{\delta}^{\sharp,+}(T^{+}f)\|_{L^{p}(w)}\\
 & \leq\|M^{+}f\|_{L^{p}(w)}
\end{align*}
and we are done.
\end{proof}
Recall that the commutator of a linear operator $T$ and a locally
integrable function $b$ is defined as 
\[
[b,T]f(x)=b(x)Tf(x)-T(bf)(x).
\]
The iterated commutator $T_{b}^{k}$ consists precisely in iterating
the commutator.
\[
T_{b}^{k}f(x)=[b,T_{b}^{k-1}]f(x)
\]
where $T_{b}^{0}f(x)=Tf(x)$. For the commutator and the iterated
commutator of $b\in BMO$ and a Calderón-Zygmund one-sided singular
integral $T^{+}$ again in \cite[Lemma 1]{LR}, it was shown that
for $0<\delta<\gamma<1$, we have that
\begin{equation}
M_{\delta}^{\sharp,+}((T^{+})_{b}^{k}f)(x)\lesssim\sum_{j=0}^{k-1}\|b\|_{BMO}^{k-j}M_{\gamma}((T^{+})_{b}^{j})f(x)+\|b\|_{BMO}^{k}(M^{+})^{k+1}f(x).\label{eq:MsharpComm}
\end{equation}
Relying upon that pointwise estimate we have the following Theorem.
\begin{thm}
\label{thm:comm}Let $0<p<\infty$ and $\varepsilon>0$ and assume
that $w\in C_{\max\{p,1\}+\varepsilon}^{+}.$ Then 
\[
\|(T^{+})_{b}^{k}f\|_{L^{p}(w)}\lesssim\|b\|_{BMO}^{k}\|(M^{+})^{k+1}f\|_{L^{p}(w)}.
\]
\end{thm}

\begin{proof}
Observe that it suffices to show that for $\delta_{1}\in(0,1)$ such
that $1<\frac{p}{\delta_{1}}<\max\{p,1\}+\varepsilon$ the following
holds
\begin{equation}
\|M_{\delta_{1}}^{+}(T^{+})_{b}^{k}f\|_{L^{p}(w)}\lesssim\|b\|_{BMO}^{k}\|(M^{+})^{k+1}f\|_{L^{p}(w)}.\label{eq:HI}
\end{equation}
We proceed by induction. Assume first that $k=1$. Let $0<\delta_{1}<\delta_{2}<1$
such that $1<\frac{p}{\delta_{i}}<\max\{p,1\}+\varepsilon$. We have
that, taking into account (\ref{eq:MsharpComm}) and Lemma \ref{Lem:YabutaDelta},
\begin{align*}
\|M_{\delta_{1}}^{+}(T^{+})_{b}^{1}f\|_{L^{p}(w)} & \lesssim\left\Vert M_{\delta_{2}}^{\sharp,+}\left((T^{+})_{b}^{1}f\right)\right\Vert _{L^{p}(w)}\\
 & \lesssim\|b\|_{BMO}\|M_{\delta_{2}}^{+}(T^{+}f)\|_{L^{p}(w)}+\|b\|_{BMO}\|(M^{+})^{2}f\|_{L^{p}(w)}\\
 & \lesssim\|b\|_{BMO}\|M^{+}f\|_{L^{p}(w)}+\|b\|_{BMO}\|(M^{+})^{2}f\|_{L^{p}(w)}\\
 & \lesssim\|b\|_{BMO}\|(M^{+})^{2}f\|_{L^{p}(w)}
\end{align*}
where the estimate for $\|M_{\delta_{2}}(T^{+}f)\|_{L^{p}(w)}$ follows
by the same argument provided in the proof of Theorem \ref{thm:CZO}.

Assume now that (\ref{eq:HI}) holds for $1,2,\dots k-1$. Let $0<\delta_{1}<\delta_{2}<1$
such that $1<\frac{p}{\delta_{i}}<\max\{p,1\}+\varepsilon$. Then
, again by (\ref{eq:MsharpComm}) and Lemma \ref{Lem:YabutaDelta},
\begin{align*}
\|M_{\delta_{1}}^{+}(T^{+})_{b}^{k}f\|_{L^{p}(w)} & \lesssim\left\Vert M_{\delta_{1}}^{\sharp,+}\left((T^{+})_{b}^{k}f\right)\right\Vert _{L^{p}(w)}\\
 & \lesssim\sum_{j=0}^{k-1}\|b\|_{BMO}^{k-j}\|M_{\delta_{2}}^{+}((T^{+})_{b}^{j})f(x)\|_{L^{p}(w)}+\|b\|_{BMO}^{k}\|(M^{+})^{k+1}f\|_{L^{p}(w)}.\\
 & \lesssim\sum_{j=0}^{k-1}\|b\|_{BMO}^{k}\|(M^{+})^{j}f(x)\|_{L^{p}(w)}+\|b\|_{BMO}^{k}\|(M^{+})^{k+1}f\|_{L^{p}(w)}\\
 & \lesssim\|b\|_{BMO}^{k}\|(M^{+})^{k+1}f\|_{L^{p}(w)}
\end{align*}
and we are done.
\end{proof}
With analogous arguments relying upon the corresponding pointwise
sharp inequality we may settle the following result. We recall that
if $\bar{A}$ is a Young function and $T^{+}$ is an operator associated
to a kernel $K\in H_{\bar{A}}$ with support in $(-\infty,0)$, then
by 
\[
M_{\delta}^{\sharp,+}((T^{+}f)(x)\lesssim M_{A}^{+}f(x)
\]
and if, $A$ and $B$ are Young functions, $\overline{C}^{-1}(t)=e^{t^{\frac{1}{k}}}$
with $k$ a positive integer such that $A^{-1}(t)B^{-1}(t)\overline{C}^{-1}(t)\leq t$
for $t\geq1$ and $K\in H_{B}\cap H_{\bar{A},k}$ then, for $0<\delta<\gamma<1$,
\[
M_{\delta}^{\sharp,+}((T^{+})_{b}^{k}f)(x)\lesssim\sum_{j=0}^{k-1}\|b\|_{BMO}^{k-j}M_{\gamma}^{+}((T^{+})_{b}^{j})f(x)+\|b\|_{BMO}^{k}M_{A}^{+}f(x).
\]

Arguing as above, we have the following results.
\begin{thm}
Let $A$ be a Young function and assume that $K\in\mathcal{H}_{\overline{A}}$
. Let $0<p<\infty$ and $\varepsilon>0$ and assume that $w\in C_{\max\{p,1\}+\varepsilon}^{+}.$
Then
\[
\|T^{+}f\|_{L^{p}(w)}\lesssim\|M_{A}^{+}f\|_{L^{p}(w)}.
\]
\end{thm}

\begin{thm}
Let $k$ be a positive integer and assume that $A$ and $B$ are Young
functions, and $A^{-1}(t)B^{-1}(t)\overline{C}^{-1}(t)\leq t$ for
$t\geq1$ where $\overline{C}^{-1}(t)=e^{t^{\frac{1}{k}}}$ . Assume
also that $K\in\mathcal{H}_{B}\cap\mathcal{H}_{\bar{A},k}$ and that
$b\in BMO$. Then, if $0<p<\infty$, $\varepsilon>0$ and $w\in C_{\max\{p,1\}+\varepsilon}^{+}$,
we have that
\[
\|(T^{+})_{b}^{k}f\|_{L^{p}(w)}\lesssim\|b\|_{BMO}^{k}\|M_{A}^{+}f\|_{L^{p}(w)}.
\]
\end{thm}

We would like to end the section providing a result for maximal singular
integral operators.
\begin{thm}
\label{thm:Max}Let $A$ be a Young function and assume that $K\in\mathcal{H}_{\overline{A}}$
. Let $0<p<\infty$ and $\varepsilon>0$ and assume that $w\in C_{\max\{p,1\}+\varepsilon}^{+}.$
Then
\[
\|(T^{+})^{*}f\|_{L^{p}(w)}\lesssim\|M_{A}^{+}f\|_{L^{p}(w)}.
\]
\end{thm}

Before settling this result, observe that if $K$ satisfies (\ref{eq:HLipschitz})
in particular $K\in\mathcal{H}_{\infty}$, and consequently, the preceding
result recovers the main result in \cite{RdT}.
\begin{proof}[Proof of Theorem \ref{thm:Max}]
Observe that by the Cotlar type inequality in Theorem \ref{thm:CotlarIneq},
we have that for any $\delta\in(0,1)$. 
\[
\|(T^{+})^{*}f\|_{L^{p}(w)}\lesssim\|M_{\delta}^{+}(T^{+}f)\|_{L^{p}(w)}+\|M_{A}^{+}f\|_{L^{p}(w)}
\]
and hence it suffices to deal with the first term. An analogous argument
to the one provided to settle Theorem \ref{thm:CZO}, choosing a suitable
$\delta$, shows that 
\[
\|M_{\delta}^{+}(T^{+}f)\|_{L^{p}(w)}\lesssim\|M_{A}^{+}f\|_{L^{p}(w)}
\]
and we are done.
\end{proof}

\subsection{The differential transform operator}

Given $\{v_{j}\}\in\ell^{\infty}$ we define
\[
T^{+}f(x)=\sum_{j\in\mathbb{Z}}v_{j}(D_{j}f(x)-D_{j-1}f(x)),\qquad D_{j}f(x)=\frac{1}{2^{j}}\int_{x}^{x+2^{j}}f(t)dt.
\]
As the authors point out in \cite{LMRdT} this operator, that was
previously studied in \cite{BLMRMdTT,JR}, arises when studying the
rate of convergence of the averages $D_{j}f$. Note that $D_{j}f\rightarrow f$
a.e. when $j\rightarrow-\infty$ and that $D_{j}f\rightarrow0$ when
$j\rightarrow\infty$ for appropriate $f$.

Observe that $T^{+}$ is a one-sided singular integral since $T^{+}f=K*f$
for $K$ supported on $(-\infty,0)$ and defined as
\[
K(x)=\sum_{j\in\mathbb{Z}}v_{j}\left(\frac{1}{2^{j}}\chi_{(-2^{j},0)}(x)-\frac{1}{2^{j-1}}\chi_{(-2^{j-1},0)}(x)\right).
\]
As it was stated in \cite[Remark 4.11]{LMRdT}, it is possible to
show that $K\in\mathcal{H}_{A,k}$ with $A(t)=\exp\left(\frac{t^{\frac{1}{1+k}}}{(\log t)^{\frac{1+\varepsilon}{1+k}}}\right)$,
and hence by techniques in \cite{LMRdT} we have that if $b\in BMO$
and $k$ is a non-negative integer, for $0<\delta<\gamma<1$, 
\[
M_{\delta}^{\sharp,+}((T^{+})_{b}^{k}f)(x)\lesssim\sum_{j=0}^{k-1}\|b\|_{BMO}^{k-j}M_{\gamma}^{+}((T^{+})_{b}^{j})f(x)+\|b\|_{BMO}^{k}M_{L(\log L)^{1+k}(\log\log L)^{1+\varepsilon}}^{+}f(x).
\]
where the first term is interpreted as 0 if $k=0$. Then, arguing
as in the preceding section we have the following result.
\begin{thm}
Let $k$ be a non negative integer. Then, if $0<p<\infty$, $\varepsilon>0$
and $w\in C_{\max\{p,1\}+\varepsilon}^{+}$, we have that
\[
\|(T^{+})_{b}^{k}f\|_{L^{p}(w)}\lesssim\|b\|_{BMO}^{k}\|M_{L(\log L)^{1+k}(\log\log L)^{1+\varepsilon}}^{+}f\|_{L^{p}(w)}.
\]
\end{thm}

\subsection{The one-sided discrete square function and its commutator}

We recall that the one-sided discrete square function is defined as
follows. If $f$ is locally integrable in $\mathbb{R}$ and $s>0$
we consider the averages
\[
A_{s}f(x)=\frac{1}{s}\int_{x}^{x+s}f(y)dy.
\]
Hence the one-sided discrete square function of $f$ is given by
\[
S^{+}f(x)=\left(\sum_{n\in\mathbb{Z}}|A_{2^{n}}f(x)-A_{2^{n-1}}f(x)|^{2}\right)^{\frac{1}{2}}.
\]
This operator was studied in \cite{dTT} and \cite{LRdT1}. In \cite{LRdT2}
the authors deal with the following operator
\[
\mathcal{O}^{+}f(x)=\left(\sum_{n\in\mathbb{Z}}\sup_{s\in[2^{n},2^{n+1})}|A_{2^{n}}f(x)-A_{s}f(x)|^{2}\right)^{\frac{1}{2}},
\]
which dominates pointwise $S^{+}f$, and show that if $0<\delta<1$
\[
M_{\delta}^{\sharp,+}(\mathcal{O}^{+}f)(x)\lesssim M_{L\log L}^{+}f(x).
\]
This fact allows them to settle the corresponding Coifman-Fefferman
estimate. Here, arguing as we did to settle Theorem \ref{thm:CZO},
we have the following result.
\begin{thm}
Let $0<p<\infty$ and $\varepsilon>0$ and assume that $w\in C_{\max\{p,1\}+\varepsilon}^{+}.$
Then 

\[
\|Gf\|_{L^{p}(w)}\lesssim\|M_{L\log L}^{+}f\|_{L^{p}(w)}
\]
where $G$ stands either for $S^{+}$ or for $\mathcal{O}^{+}$.
\end{thm}

Further assuming that $b\in BMO$ the authors also study the commutators
associated to the operators above. They show that, for $0<\delta<\gamma<1$,
\[
M_{\delta}^{\sharp,+}((\mathcal{O}^{+})_{b}^{k}f)(x)\lesssim\sum_{j=0}^{k-1}\|b\|_{BMO}^{k-j}M_{\gamma}^{+}((\mathcal{O}^{+})_{b}^{j})f(x)+M_{L(\log L)^{1+k}}^{+}f(x).
\]
As a consequence we can derive the following result.
\begin{thm}
Let $0<p<\infty$ and $\varepsilon>0$ and assume that $w\in C_{\max\{p,1\}+\varepsilon}^{+}$
and that $b\in BMO$. If $k$ is a positive integer, then 

\[
\|G{}_{b}^{k}f\|_{L^{p}(w)}\lesssim\|b\|_{BMO}^{k}\|M_{L(\log L)^{1+k}}^{+}f\|_{L^{p}(w)}
\]
where $G$ stands either for $S^{+}$ or for $\mathcal{O}^{+}$.
\end{thm}

\subsection{Riemann-Liouville and Weyl Fractional Integral Operators and their
commutators}

We recall that given $0<\alpha<1$ and locally integrable functions
$f$ and $b$, the Weyl fractional integral and its commutators are
defined as
\begin{align*}
I_{\alpha}^{+}f(x) & =\int_{x}^{\infty}\frac{f(y)}{(y-x)^{1-\alpha}}dy\qquad\text{ and }\\
(I_{\alpha}^{+})_{b}^{k}f(x) & =\int_{x}^{\infty}(b(x)-b(y))^{k}\frac{f(y)}{(y-x)^{1-\alpha}}dy
\end{align*}
respectively. Analogously we define the Riemann-Liouville fractional
integral and its commutators as 
\begin{align*}
I_{\alpha}^{-}f(x) & =\int_{-\infty}^{x}\frac{f(y)}{(x-y)^{1-\alpha}}dy\qquad\text{and}\\
(I_{\alpha}^{-})_{b}^{k}f(x) & =\int_{-\infty}^{x}(b(x)-b(y))^{k}\frac{f(y)}{(x-y)^{1-\alpha}}dy.
\end{align*}
In \cite{BL} it was shown that for every non-negative $k$, if $b\in BMO$
and $0<\delta<\gamma<1$, 
\[
M_{\delta}^{\sharp,+}((I_{\alpha}^{+})_{b}^{k}f)(x)\lesssim\sum_{j=0}^{k-1}\|b\|_{BMO}^{k-j}M_{\gamma}^{+}((I_{\alpha}^{+})_{b}^{j})f(x)+\|b\|_{BMO}^{k}M_{(\alpha),L(\log L)^{k}}^{+}f(x)
\]
where
\[
M_{(\alpha),L(\log L)^{k}}^{+}f(x)=\sup_{h>0}h^{\alpha}\|f\|_{L(\log L)^{k},(x,x+h)}
\]
and the first term in the right hand side is interpreted as $0$ if
$k=0$. Relying upon that $M_{\delta}^{\sharp,+}$ estimate and arguing
as in the proofs of Theorems \ref{thm:CZO} and \ref{thm:comm} it
is possible to settle the following result.
\begin{thm}
Let $k$ be a non negative integer. Then, if $0<p<\infty$, $\varepsilon>0$
and $w\in C_{\max\{p,1\}+\varepsilon}^{+}$ we have that
\[
\|(I_{\alpha}^{+})_{b}^{k}f\|_{L^{p}(w)}\lesssim\|b\|_{BMO}^{k}\|M_{(\alpha),L(\log L)^{k}}^{+}\|_{L^{p}(w)}.
\]
\end{thm}

\section{\label{sec:Proof-MT}Proof of Theorem \ref{thm:OneSidedYabuta}}

\subsection{Sufficiency}

To settle the sufficiency part in Theorem \ref{thm:OneSidedYabuta}
we need to borrow two Lemmas from \cite{RdT}. The first one is the
following.
\begin{lem}
\label{lem:1}Assume that $w\in C_{q}^{+}$ with $1<q<\infty$. Then,
for any $\delta>0$ there exists $c(\delta)$ such that for any disjoint
family of intervals $J_{j}$ contained in $I=(a,b)$ we have that
\[
\int_{I}\sum_{j}(M^{+}\chi_{J_{j}})^{q}w\leq c(\delta)w(I)+\delta\int_{\mathbb{R}}(M^{+}\chi_{I})^{q}w
\]
and 
\[
\int_{\mathbb{R}}\sum_{j}(M^{+}\chi_{J_{j}})^{q}w\lesssim\int_{\mathbb{R}}(M^{+}\chi_{I})^{q}w.
\]
\end{lem}

To state the next Lemma we need to define a new operator, $M_{p,q}^{+}$.
Let $f$ be an nonnegative measurable function. Let us consider
\[
\Omega_{k}=\left\{ x\in\mathbb{R}\,:\,f(x)>2^{k}\right\} =\bigcup_{i}I_{i}^{k}
\]
where $I_{i}^{k}$ are the connected components of $\Omega_{k}$.
Then 
\[
(M_{p,q}^{+}f(x))^{p}=\sum_{i,k}2^{pk}(M^{+}(\chi_{I_{i}^{k}})(x))^{q}.
\]
Having this definition at our disposal we present the second Lemma
we borrow from \cite{RdT}.
\begin{lem}
\label{lem:2}Let $1<p<q<\infty,$ $w\in C_{q}^{+}$ and $f$ non-negative,
bounded and of compact support. Then 
\[
\int_{\mathbb{R}}M_{p,q}^{+}(M^{+}f)w\lesssim\int_{\mathbb{R}}(M^{+}f)^{p}w.
\]
\end{lem}

Having those Lemmas at our disposal we are in the position to settle
Theorem \ref{thm:OneSidedYabuta}.
\begin{proof}[Proof of Theorem \ref{thm:OneSidedYabuta}]
Let $\Omega_{k}=\{x\,:\,M^{+}f(x)>2\cdot2^{k}\}=\bigcup_{j}J_{j}^{k}$
where $J_{j}^{k}$ are the connected components of $\Omega_{k}$.
Let us fix $(a,b)=J_{j}^{k}$. We partition $(a,b)$ as follows. Let
$x_{0}=a$ and choose $x_{i+1}$ such that $x_{i+1}-x_{i}=b-x_{i+1}$
and let $I_{i}^{k}=(x_{i},x_{i+1})$. By the good-$\lambda$ inequality
established in \cite[Theorem 4]{MRdT}, we have that 
\[
|E_{i}^{k}|=\left|\left\{ x\in I_{i}^{k}\,:\,M^{+}f(x)>2^{k+1},M^{+,\sharp}f(x)\leq\gamma2^{k}\right\} \right|\leq C\gamma|I_{i+1}^{k}|\qquad0<\gamma<1.
\]
From the $C_{q}^{+}$ condition it follows that 
\[
w(E_{i}^{k})\leq C\gamma^{\varepsilon}\int_{\mathbb{R}}\left(M^{+}\chi_{I_{i}^{k}\cup I_{i+1}^{k}}\right)^{q}w.
\]
Summing on $i$ and taking into account Lemma \ref{lem:1} we have
that 
\begin{align*}
 & w\left(\left\{ x\in J_{j}^{k}\,:\,M^{+}f(x)>2^{k+1},M^{\sharp,+}f(x)\leq\gamma2^{k}\right\} \right)\\
 & \leq C\gamma^{\varepsilon}\sum_{i}\int_{\mathbb{R}}\left(M^{+}\chi_{I_{i}^{k}\cup I_{i+1}^{k}}\right)^{q}w\lesssim C\gamma^{\varepsilon}\int_{\mathbb{R}}\left(M^{+}\chi_{J_{j}^{k}}\right)^{q}w.
\end{align*}
Now summing over all $j$, 
\begin{align*}
 & w\left(\left\{ x\in\Omega_{k}\,:\,M^{+}f(x)>2^{k+1},M^{+,\sharp}f(x)\leq\gamma2^{k}\right\} \right)\\
\leq & C\gamma^{\varepsilon}\sum_{j}\int_{\mathbb{R}}\left(M^{+}\chi_{J_{j}^{k}}\right)^{q}w.
\end{align*}
Having those estimates at our disposal we can argue as follows.
\begin{align*}
 & \int_{\mathbb{R}}(M^{+}f(x))^{p}w(x)dx\lesssim\sum_{k\in\mathbb{Z}}2^{kp}w(\Omega_{k})\\
\leq & \sum_{k\in\mathbb{Z}}2^{kp}w\left(\left\{ x\in\Omega_{k}\,:\,M^{+}f(x)>2^{k+1},M^{+,\sharp}f(x)\leq\gamma2^{k}\right\} \right)\\
 & +\sum_{k\in\mathbb{Z}}2^{kp}w\left(\left\{ x\in\mathbb{R}\,:\,M^{+,\sharp}f(x)>\gamma2^{k}\right\} \right)\\
\leq & C\gamma^{\varepsilon}\sum_{k\in\mathbb{Z}}2^{kp}\sum_{j}\int_{\mathbb{R}}\left(M^{+}\chi_{J_{j}^{k}}\right)^{q}w+\sum_{k\in\mathbb{Z}}2^{kp}w\left(\left\{ x\in\mathbb{R}\,:\,M^{+,\sharp}f(x)>\gamma2^{k}\right\} \right)\\
\leq & C\gamma^{\varepsilon}\int_{\mathbb{R}}M_{p,q}^{+}(M^{+}f)(x)w(x)dx+c_{\gamma}\int_{\mathbb{R}}(M^{+,\sharp}f(x))^{p}w(x)dx\\
\leq & C\gamma^{\varepsilon}\int_{\mathbb{R}}(M^{+}f(x)){}^{p}w(x)dx+c_{\gamma}\int_{\mathbb{R}}(M^{+,\sharp}f(x))^{p}w(x)dx
\end{align*}
where in the last step we have used Lemma \ref{lem:2}. 

Observe that if $\int_{\mathbb{R}}\left(M^{+}f(x)\right){}^{p}w(x)dx<\infty$,
choosing $\gamma$ small enough the desired estimate follows. We end
the proof observing that for $f\in L_{c}^{\infty}$
\[
\int_{\mathbb{R}}\left(M^{+,\sharp}f(x)\right){}^{p}w(x)dx<\infty
\]
implies 
\[
\int_{\mathbb{R}}\left(M^{+}f(x)\right){}^{p}w(x)dx<\infty.
\]
Indeed, note that since $f\in L_{c}^{\infty}$ we may assume that
$\supp f\subset[a,b]$. For $x>b$ we have that $M^{+,\sharp}f(x)=M^{+}f(x)=0$,
and for $x\rightarrow-\infty$, $M^{+}f(x)\simeq M^{+,\sharp}f(x)\simeq\frac{1}{|x|}.$
\end{proof}

\subsection{Necessity}

The proof of the necessity will rely upon the following Lemma, which
is a one-sided version of some of the results in \cite{M,S}.
\begin{lem}
\label{lem:equivCp}If for every $a<b<c$ with $c-b<b-a$ and $E\subset(a,b)$
\begin{equation}
w(E)\lesssim\frac{1}{\left[1+\log^{+}\left(\frac{|(b,c)|}{|E|}\right)\right]^{p}}\int_{\mathbb{R}}M^{+}(\chi_{(a,c)})^{p}w(x)dx\label{eq:CpEquiv}
\end{equation}
then $w\in C_{p}^{+}$.
\end{lem}

Before settling the Lemma we show how to derive from it the necessity
in Theorem \ref{eq:HI}. 
\begin{proof}[Proof of the necessity in Theorem \ref{eq:HI}]
Assume that for a certain $1<p<\infty$ and a weight $w$, 
\[
\|M^{+}f\|_{L^{p}(w)}\lesssim\|M^{\sharp,+}f\|_{L^{p}(w)}.
\]
Let $I=(a,c)$ an interval. Let $a<b<c$. Assume that $E$ is a measurable
set contained in $(a,b)$. Let us define
\begin{align*}
f(x) & =\log^{+}\left(\frac{|(b,c)|}{|E|}M^{-}(\chi_{E})(x)\right)+\chi_{I}(x)\\
 & =g(x)+\chi_{I}(x).
\end{align*}
Arguing as in \cite{S} and \cite{Y} we have that 
\begin{align}
 & \frac{1}{|(a,c)|}\int_{(a,c)}\log^{+}\left(\frac{|(b,c)|}{|E|}M^{-}(\chi_{E})\right)\lesssim1\label{eq:(1Y)}\\
 & \|f\|_{BMO^{+}}\lesssim1\label{eq:(2Y)}\\
 & f(x)=\log^{+}\left(\frac{|(b,c)|}{|E|}\right)+1\qquad\text{a.e.}\quad x\in E.\label{eq:(3Y)}
\end{align}
Observe that $M^{-}(\chi_{E})\leq\frac{|E|}{\dist(x,E)}$ for $x\not\in E$.
In particular, if $x\geq c$ we have that
\[
M^{-}(\chi_{E})(x)\leq\frac{|E|}{\dist(x,E)}\leq\frac{|E|}{x-b}\leq\frac{|E|}{c-b}.
\]
This yields $g(x)=0$ if $x\geq c$. On the other hand, if $x\leq a$
then also $M^{-}(\chi_{E})(x)=0$. And consequently $g(x)=0$. Hence
$\supp f\subset(a,c)$. Now we observe that if $x>c$ we have that
$M^{+}f(x)=0$ and $M^{\sharp,+}f(x)=0$. If $x<c$ we have two cases.
If $x\in(a-|I|,c)$ then, by (\ref{eq:(2Y)}) 
\[
M^{\sharp,+}f(x)\lesssim\|f\|_{BMO^{+}}\lesssim M^{+}(\chi_{(a,c)})(x).
\]
If $x<a-|I|$ then$M^{+}\chi_{(a,c)}(x)=\frac{c-a}{c-x}$ and $c-x\geq2|I|$,
from which it follows that $a-x=c-x-|I|\geq\frac{c-x}{2}$. Hence,
we have that by (\ref{eq:(1Y)})
\begin{align*}
M^{\sharp,+}f(x) & \lesssim M^{+}f(x)\leq M^{+}g(x)+M^{+}(\chi_{(a,c)})(x)\\
 & \leq\frac{\int_{a}^{c}g(y)dy}{a-x}+M^{+}(\chi_{(a,c)})(x)\\
 & \lesssim\frac{c-a}{c-x}\frac{1}{c-a}\int_{a}^{c}g(y)dy+M^{+}(\chi_{(a,c)})(x)\\
 & \lesssim\frac{c-a}{c-x}+M^{+}(\chi_{(a,c)})(x)\lesssim M^{+}(\chi_{(a,c)})(x).
\end{align*}
Gathering the estimates above we have that 
\[
M^{\sharp,+}f(x)\lesssim M^{+}(\chi_{(a,c)})(x)\qquad x\in\mathbb{R}.
\]
Taking into account the preceding estimate and (\ref{eq:(3Y)}), we
have that
\begin{align*}
w(E) & =\frac{1}{\left[1+\log^{+}\left(\frac{|(b,c)|}{|E|}\right)\right]^{p}}\left[1+\log^{+}\left(\frac{|(b,c)|}{|E|}\right)\right]^{p}\int_{E}w(x)dx\\
 & =\frac{1}{\left[1+\log^{+}\left(\frac{|(b,c)|}{|E|}\right)\right]^{p}}\int_{E}|f(x)|^{p}w(x)dx\\
 & \leq\frac{1}{\left[1+\log^{+}\left(\frac{|(b,c)|}{|E|}\right)\right]^{p}}\int_{\mathbb{R}}(M^{+}f(x))^{p}w(x)dx\\
 & \lesssim\frac{1}{\left[1+\log^{+}\left(\frac{|(b,c)|}{|E|}\right)\right]^{p}}\int_{\mathbb{R}}|M^{\sharp,+}f(x)|^{p}w(x)dx\\
 & \lesssim\frac{1}{\left[1+\log^{+}\left(\frac{|(b,c)|}{|E|}\right)\right]^{p}}\int_{\mathbb{R}}(M^{+}\chi_{(a,c)})^{p}w(x)dx
\end{align*}
and we are done.
\end{proof}
We devote the remainder of the section to settle Lemma \ref{lem:equivCp}.
First we will need the following Lemma.
\begin{lem}
Let $1<p<\infty$. Assume that $w$ is a weight such that
\[
\int_{E}w\lesssim\frac{1}{\left(1+\log^{+}\left(\frac{c-b}{|E|}\right)\right)^{p}}\int_{-\infty}^{\infty}\left(M^{+}\chi_{(a,c)}\right)^{p}w
\]
for every $a<b<c$ with $c-b<b-a$ and where $E$ is any measurable
set contained in $(a,b)$. Then, for every family $\{I_{k}\}_{k=1}^{n}$
of disjoint subintervals of an interval $I$, if we denote
\[
\Delta(x)=\sum_{k=1}^{n}M^{+}(\chi_{I_{k}})^{p}(x),
\]
then 
\begin{equation}
\int_{\mathbb{R}}\Delta w\lesssim\frac{1}{\left(1+\log\left(\frac{|I|}{\sum_{k=1}^{n}|I_{k}|}\right)\right)^{p-1}}\int_{\mathbb{R}}\left(M^{+}\chi_{I}\right)^{p}w.\label{eq:Conclusion}
\end{equation}
\end{lem}

\begin{proof}
Let $I=(a,c)$ be an interval and $\{I_{k}\}$ a family of disjoint
subintervals of $I$. First we note that 
\[
M^{+}(\chi_{I_{k}})(x)=0
\]
for each $x>c$. Then we have that 
\[
\int_{\mathbb{R}}\Delta w=\int_{-\infty}^{c}\Delta w=\int_{a'}^{c}\Delta w+\int_{-\infty}^{a'}\Delta w.
\]
where $a'=a-|I|$. First we deal with the second term. Observe that
if $d_{k}$ is the right endpoint of $I_{k}$, then
\begin{align*}
\int_{-\infty}^{a'}\Delta w & =\int_{-\infty}^{a'}\sum_{k=1}^{n}\left(M^{+}\chi_{I_{k}}\right)^{p}(x)w(x)dx\leq\int_{-\infty}^{a'}\sum_{k=1}^{n}\left(\frac{|I_{k}|}{d_{k}-x}\right)^{p}w(x)dx\\
 & \leq\int_{-\infty}^{a'}\left(\sum_{k=1}^{n}\frac{|I_{k}|}{d_{k}-x}\right)^{p}w(x)dx\lesssim\int_{-\infty}^{a'}\left(\sum_{k=1}^{n}\frac{|I_{k}|}{c-x}\right)^{p}w(x)dx\\
 & \leq\int_{-\infty}^{a}\left(\frac{|I|}{c-x}\right)^{p}w(x)dx\lesssim\int_{-\infty}^{c}(M^{+}\chi_{I})^{p}w(x)dx.
\end{align*}
Now we deal with the first term. Let $j$ be the least integer such
that 
\[
\log\left(\frac{\sum_{k=1}^{n}|I_{k}|}{|I|}\right)\leq j
\]
and $J$ the least integer such that for some $D\in(0,1)$ to be chosen
later,
\[
\log\left(\frac{1}{D}\log\left(\frac{e|I|}{\sum_{k=1}^{n}|I_{k}|}\right)\right)\leq J.
\]
Note that since $j\leq0$ and $J>0$ we have that $j<J$. Let 
\begin{align*}
Q & =\{\Delta(x)\leq e^{j}\}\\
S & =\{e^{j}<\Delta(x)\leq e^{J}\}\\
T & =\{\Delta(x)>e^{J}\}
\end{align*}
For $Q$ we have that
\begin{align*}
\int_{(a',c)\cap Q}\Delta(x)w(x)dx & \leq\int_{(a',c)\cap Q}e^{j}w(x)dx\\
 & =e\int_{(a',c)\cap Q}e^{j-1}w(x)dx\\
 & \leq e\int_{(a',c)\cap Q}e^{\log\left(\frac{\sum_{k=1}^{n}|I_{k}|}{|I|}\right)}w(x)dx\\
 & \leq e\frac{\sum_{k=1}^{n}|I_{k}|}{|I|}\int_{(a',c)}w(x)dx
\end{align*}
and the right hand side is bounded by the right hand side of (\ref{eq:Conclusion}).
We continue with $S$. 
\begin{align*}
\int_{(a',c)\cap S}\Delta(x)w(x)dx & \leq\sum_{k=j}^{J-1}e^{k+1}\int_{\{\Delta(x)>e^{k}\}\cap(a',c)}w(x)dx\\
 & =\sum_{k=j}^{J-1}e^{k+1}w(E_{e^{k}})
\end{align*}
where 
\[
E_{\lambda}=\{\Delta(x)>\lambda\}\cap(a',c).
\]
To continue with the argument, we borrow ideas from. \cite[p.  406]{RdT}.
We begin noting that there exists $B>1$ and that we can choose the
$D\in(0,1)$ above in such a way that
\[
|E_{\lambda}|\leq Be^{-D\lambda}|(a',c)|.
\]
At this point we use our hypothesis on the weight $w$. For that purpose
we split $(a',c)$ as follows. Let $x_{0}=a'$ and let us define recursively
$x_{i}-x_{i-1}=c-x_{i}$. Associated to the collection of intervals
$(x_{i},x_{i+1})$ we consider the sets 
\[
E_{\lambda}^{i}=E_{\lambda}\cap(x_{i},x_{i+1}).
\]
Observe that for each $i$ we may assume that the elements that we
consider in the sum $\Delta(x)$ are contained in $(x_{i},c)$ (the
remaining terms are zero). Hence we have that 
\[
|E_{\lambda}^{i}|\leq Be^{-D\lambda}|(x_{i},c)|=4Be^{-D\lambda}(x_{i+2}-x_{i+1}).
\]
Now we use the hypothesis for $x_{i}$, $x_{i+1}$, $x_{i+2}$ and
we have that
\begin{align*}
w(E_{\lambda}^{i}) & \lesssim\frac{1}{\left(1+\log^{+}\left(\frac{x_{i+2}-x_{i+1}}{|E_{\lambda}^{i}|}\right)\right)^{p}}\int_{-\infty}^{\infty}(M^{+}\chi_{(x_{i},x_{i+2})})^{p}w\\
 & \lesssim\frac{1}{\left(1+\log^{+}\left(\frac{1}{4Be^{-D\lambda}}\right)\right)^{p}}\int_{-\infty}^{\infty}(M^{+}\chi_{(x_{i},x_{i+2})})^{p}w.
\end{align*}
Summing in $i$ we have that it follows from the definition of the
partition $x_{i}$, which leads to a geometric series, that
\[
\sum_{i}\int_{-\infty}^{\infty}\left(M^{+}(\chi_{(x_{i},x_{i+2})})\right)^{p}w\leq C\int_{-\infty}^{\infty}\left(M^{+}\chi_{(a',c)}\right)^{p}w.
\]
Hence, 
\begin{align*}
w(E_{\lambda}) & \lesssim\frac{1}{\left(1+\log^{+}\left(\frac{1}{4Be^{-D\lambda}}\right)\right)^{p}}\int_{-\infty}^{\infty}\left(M^{+}\chi_{(a',c)}\right)^{p}w\\
 & \lesssim\frac{1}{\left(1+\log^{+}\left(\frac{1}{4Be^{-D\lambda}}\right)\right)^{p}}\int_{-\infty}^{\infty}\left(M^{+}\chi_{(a,c)}\right)^{p}w
\end{align*}
and we have that 
\begin{align*}
\int_{(a',c)\cap S}\Delta(x)w(x)dx & \leq\sum_{k=j}^{J}e^{k+1}w(E_{e^{k}})\\
 & \lesssim\sum_{k=j}^{J}\frac{1}{\left(1+\log^{+}\left(\frac{e^{De^{k}}}{4B}\right)\right)^{p}}e^{k+1}\int_{-\infty}^{\infty}\left(M^{+}\chi_{(a,c)}\right)^{p}w
\end{align*}
and it suffices to estimate the sum. We proceed as follows.
\begin{align*}
 & \sum_{k=j}^{J}\frac{1}{\left(1+\log^{+}\left(\frac{e^{De^{k}}}{4B}\right)\right)^{p}}e^{k+1}\\
 & \leq\sum_{k=j}^{\left\lfloor \log\left(\frac{1}{D}\log(4B)\right)\right\rfloor }e^{k+1}+\sum_{k=\left\lfloor \log\left(\frac{1}{D}\log(4B)\right)\right\rfloor +1}^{J}\frac{1}{\left(1+De^{k}-\log(4B)\right)^{p}}e^{k+1}\\
 & \lesssim e^{j}\frac{(1-e^{\log\left(\frac{1}{D}\log(4B)\right)})}{1-e}+\sum_{k=\left\lfloor \log\left(\frac{1}{D}\log(4B)\right)\right\rfloor +1}^{J}\frac{1}{\left(De^{k}+1-De^{\log\left(\frac{1}{D}\log(4B)\right)}\right)^{p}}e^{k+1}\\
 & \lesssim e^{j}\frac{\frac{1}{D}\log(4B)-1}{e-1}+\sum_{k=\left\lfloor \log\left(\frac{1}{D}\log(4B)\right)\right\rfloor +1}^{J}\frac{e}{c^{p}D^{p}e^{k(p-1)}}\\
 & \lesssim e^{j-1}\frac{1}{D}\log(4B)+\kappa e^{-J(p-1)}\\
 & \lesssim\frac{\sum_{k=1}^{n}|I_{k}|}{|I|}+\kappa\frac{1}{\left[\frac{1}{D}\log\left(\frac{e|I|}{\sum_{k=1}^{n}|I_{k}|}\right)\right]^{p-1}}
\end{align*}
and again this term is bounded by the right hand side of (\ref{eq:Conclusion}).

Finally, for $T$ we have that
\begin{align*}
\int_{T\cap(a',c)}\Delta(x)w(x)dx & \leq\sum_{i=J}^{\infty}e^{i+1}\int_{\{\Delta(x)>e^{i}\}\cap(a',c)}w(x)dx\\
 & \lesssim\sum_{i=J}^{\infty}e^{i+1}\frac{1}{\left(1+\log^{+}\left(\frac{1}{4Be^{-De^{i}}}\right)\right)^{p}}\int_{-\infty}^{\infty}\left(M^{+}\chi_{(a',c)}\right)^{p}w\\
 & \lesssim\sum_{i=J}^{\infty}\frac{1}{e^{i(p-1)}}\int_{-\infty}^{\infty}\left(M^{+}\chi_{(a,c)}\right)^{p}w\\
 & \lesssim\frac{1}{e^{J(p-1)}}\int_{-\infty}^{\infty}\left(M^{+}\chi_{(a,c)}\right)^{p}w\\
 & \lesssim\frac{1}{\left[\log\left(\frac{e|I|}{\sum_{k=1}^{n}|I_{k}|}\right)\right]^{p-1}}\int_{-\infty}^{\infty}\left(M^{+}\chi_{(a,c)}\right)^{p}w.
\end{align*}
\end{proof}
Armed with the preceding Lemma we can finally settle Lemma \ref{lem:equivCp}.
\begin{proof}[Proof of Lemma \ref{lem:equivCp}]
Let $I=(a,c)$ be an interval. Let $\delta>0$ such that if $\sum|I_{k}|\leq2\delta|I|$
then
\begin{equation}
\int_{\mathbb{R}}\Delta w\leq\frac{1}{2}\int_{\mathbb{R}}(M^{+}\chi_{I})^{p}w.\label{eq:CondDelta}
\end{equation}
Now assume that $a<b<c$ where $c-b<b-a$ and let $E\subset(a,b)$
be a measurable set. Let $n$ be the least integer such that $\delta^{n}|(a,b)|<|E|.$
Now we let $E_{j}=\{x\,:M^{+}(\chi_{E})(x)>\delta^{j}\}$ for $1\leq j\leq n$.
Let $J_{i}^{j}$ be the component intervals of $E_{j}$ and $\Delta_{j}(x)=\sum_{i}M^{+}(\chi_{J_{i}^{j}})(x){}^{p}$.
We claim that for $2\leq j\leq n$
\[
\int_{\mathbb{R}}\Delta_{j-1}(x)w(x)dx\leq\frac{1}{2}\int_{\mathbb{R}}\Delta_{j}(x)w(x)dx.
\]
 Assume by now that the claim holds. Note that since $\text{\ensuremath{\chi_{E}(x)\leq\Delta_{1}(x)}}$
then
\[
w(E)\leq\int_{\mathbb{R}}\Delta_{1}(x)w(x)dx
\]
and iterating the preceding inequality, 
\[
\int_{\mathbb{R}}\Delta_{1}(x)w(x)dx\leq\frac{1}{2^{n-1}}\int_{\mathbb{R}}\Delta_{n}(x)w(x)dx
\]
and consequently
\[
w(E)\leq2\frac{1}{2^{n}}\int_{\mathbb{R}}\Delta_{n}(x)w(x)dx.
\]
Note that by the definition of $n$, and taking into account that
$b-c<b-a$, 
\begin{align*}
\delta^{n}|(a,b)|<|E| & \iff\frac{1}{2^{n(-\log_{2}\delta)}}<\frac{|E|}{|(a,b)|}\\
 & \iff\frac{1}{2^{n}}<\left(\frac{|E|}{|(a,b)|}\right)^{\frac{1}{-\log_{2}\delta}}\leq\left(\frac{|E|}{|(b,c)|}\right)^{\frac{1}{-\log_{2}\delta}}
\end{align*}
so if we can show that 
\begin{equation}
\int_{\mathbb{R}}\Delta_{n}(x)w(x)dx\lesssim\int_{\mathbb{R}}(M^{+}\chi_{(a,c)}(x))^{p}w(x)dx\label{eq:LastStep}
\end{equation}
then
\[
w(E)\lesssim2\left(\frac{|E|}{|(b,c)|}\right)^{\frac{1}{-\log_{2}\delta}}\int_{\mathbb{R}}(M^{+}\chi_{(a,c)})^{p}w
\]
and we would be done.\\
Let us settle (\ref{eq:LastStep}). Observe that since 
\[
\frac{|E\cap(a,b)|}{|(a,b)|}=\frac{|E|}{|(a,b)|}>\delta^{n},
\]
we have that $a\in J_{i}^{n}$ for some component $J_{i}^{n}$ of
$E_{n}$. Let us call $J_{i}^{n}=(a',b')$. Observe that since $a\in J_{i}^{n}$
then $a'<a$. Observe that $M^{+}\chi_{E}(a')=\delta^{n}$ and since
$(a',b')$ is a component, for some $\varepsilon>0$, we have that
if $x\in[a'-\varepsilon,a')$, then $M^{+}\chi_{E}(x)\leq\delta^{n}.$
Note that since $E\subset(a,b)$ and $a>a'$ this yields that for
every $x<a'-\varepsilon$
\[
M^{+}\chi_{E}(x)\leq M^{+}\chi_{E}(a'-\varepsilon)\leq\delta^{n}.
\]
This yields that all the connected components are contained in the
interval $(a',b''),$ for some $b''\leq b$, since $M^{+}\chi_{E}(x)=0$
for every $x>b$. Now, that by the weak type $(1,1)$ of $M^{+}$,
combined with the definition of $n$, 
\begin{align*}
\sum_{i}|J_{i}^{n}| & =\left|\left\{ x\in\mathbb{R}\,:\,M^{+}\chi_{E}(x)>\delta^{n}\right\} \right|\leq\frac{1}{\delta^{n}}|E|\\
 & =\frac{1}{\delta}\frac{1}{\delta^{n-1}}\frac{|E|}{|(a,b)|}|(a,b)|\leq\frac{1}{\delta}|(a,b)|.
\end{align*}
Observe that since $(a',b')$ is some component $J_{i}^{n}$ in $E_{n}$
then
\[
a-a'\leq b'-a'\leq\sum_{i}|J_{i}^{n}|\leq\frac{1}{\delta}|(a,b)|.
\]
Consequently
\begin{align*}
b-a' & =b-a+a-a'\\
 & \leq b-a+\frac{1}{\delta}(b-a)=\left(1+\frac{1}{\delta}\right)(b-a).
\end{align*}
Since all the intervals $J_{i}^{n}$ are contained in $(a',b)$, by
(\ref{eq:Conclusion}) we have that 
\begin{align*}
\int_{\mathbb{R}}\Delta w & \lesssim\frac{1}{\left(1+\log\left(\frac{|(a',b)|}{\sum_{i}|J_{i}^{n}|}\right)\right)^{p-1}}\int_{\mathbb{R}}M^{+}(\chi_{(a',b)})^{p}w.\\
 & \leq\int_{\mathbb{R}}(M^{+}\chi_{(a,b)})^{p}w\leq\int_{\mathbb{R}}(M^{+}\chi_{(a,c)})^{p}w
\end{align*}
We end the proof of (\ref{eq:LastStep}) just noting that 
\[
M^{+}(\chi_{(a'',b)})\lesssim M^{+}(\chi_{(a,c)})
\]
since taking into account that $|(a'',b)|\leq\frac{1}{\delta}\left(\frac{1}{\delta}+1\right)|(a,b)|$
yields that $M^{+}(\chi_{(a'',b)})\simeq M^{+}(\chi_{(a,b)})$.

Since as we have just shown (\ref{eq:LastStep}) holds, we are left
with settling the claim. We argue as follows. Let $H$ be a component
of $E_{j}$. Note that, since $H$ is a component
\begin{equation}
\frac{|H\cap E|}{|H|}=\delta^{j}.\label{eq:HcapE/H}
\end{equation}
Our next step is to show that
\begin{equation}
H\cap E_{j-1}=\{M^{+}\chi_{H\cap E}>\delta^{j-1}\}.\label{eq:HEj-1}
\end{equation}
First, we observe that the components of $E_{j-1}$ are contained
in the components of $E_{j}$. Hence 
\[
H\cap E_{j-1}=\bigcup I_{k}^{j-1}
\]
where the $I_{k}^{j-1}$ are the components of $E_{j-1}$ contained
in $H$. Then we have that if $x\in H\cap E_{j-1}$ then $x\in I_{k}^{j-1}=(\alpha,\beta)$
for some $k$. Since $x\in E_{j-1}$, then $M^{+}(\chi_{E})(x)>\delta^{j-1}$
and 
\[
\frac{1}{|(x,\beta)|}\int_{x}^{\beta}\chi_{E}>\delta^{j-1}.
\]
Observe that $(x,\beta)\subset I_{k}^{j-1}\subset H$, so 
\[
\frac{1}{|(x,\beta)|}\int_{x}^{\beta}\chi_{E\cap H}>\delta^{j-1}.
\]
Consequently 
\[
M^{+}\chi_{H\cap E}(x)>\delta^{j-1}
\]
and this yields $\{M^{+}\chi_{H\cap E}>\delta^{j-1}\}\supset H\cap E_{j-1}$.

Now we prove the converse inclusion. Observe that if $M^{+}\chi_{H\cap E}(x)>\delta^{j-1}$
then 
\[
M^{+}\chi_{E}(x)\geq M^{+}\chi_{H\cap E}(x)>\delta^{j-1}
\]
and consequently $x\in E_{j-1}$. Now we have to see that $x\in H$.
Let us call $H=(d,e)$. Observe that $M^{+}\chi_{H\cap E}(x)=0$ for
every $x>e$, and hence $x\not\in\{M^{+}\chi_{H\cap E}>\delta^{j-1}\}$.
On the other hand, observe that since $H$ is a component of $E_{j}$
there exist some $\varepsilon<0$ such that if $x\in[d-\varepsilon,d]$
then $M^{+}\chi_{E}(x)\leq\delta^{j}$. Relying upon this fact, note
that if $x<d$ then we have that if $x\in[d-\varepsilon,d]$, then
$M^{+}\chi_{H\cap E}(x)\leq M^{+}\chi_{E}(x)\leq\delta^{j}$ and consequently
$x\not\in\{M^{+}\chi_{H\cap E}>\delta^{j-1}\}$ and if $x<d-\varepsilon$
then, $M^{+}\chi_{H\cap E}(x)\leq M^{+}\chi_{H\cap E}(d-\varepsilon)\leq\delta^{j}$
and also $x\not\in\{M^{+}\chi_{H\cap E}>\delta^{j-1}\}$. Hence 
\[
\{M^{+}\chi_{H\cap E}>\delta^{j-1}\}\subset H
\]
and we are done.

The weak type $(1,1)$ of $M^{+}$ combined with (\ref{eq:HEj-1})
yields
\[
|H\cap E_{j-1}|\leq\delta^{1-j}|E\cap H|
\]
and combining this with (\ref{eq:HcapE/H}) we have that
\[
|H\cap E_{j-1}|\leq\delta|H|.
\]
If we denote 
\[
\Delta_{\mathcal{H}}(x)=\sum_{I_{k}\in\mathcal{H}}(M^{+}\chi_{I_{k}}(x))^{p}
\]
where $\mathcal{H}$ is the set of component intervals in $H\cap E_{j-1}$,
we have that
\[
\sum_{I_{k}\in\mathcal{H}}|I_{k}|\leq|H\cap E_{j-1}|\leq\delta|H|.
\]
By the definition of $\delta$ and (\ref{eq:CondDelta}) we have that
\[
\int_{-\infty}^{\infty}\Delta_{\mathcal{H}}(x)w(x)dx\leq\frac{1}{2}\int_{-\infty}^{\infty}M^{+}(\chi_{H})^{p}w(x)dx.
\]
Adding those inequalities for all the components $H$ of $E_{j}$
gives
\[
\int_{\mathbb{R}}\Delta_{j-1}(x)w(x)\leq\frac{1}{2}\int_{\mathbb{R}}\Delta_{j}(x)w(x)dx
\]
for $2\leq j\leq n$ as we wanted to show. This ends the proof of
the Theorem.
\end{proof}

\section*{Acknowledgements}

The first and the second authors were partially supported by Ministerio
de Economía y Competitividad, Spain, grant PGC2018-096166-B-100 and
by Junta de Andalucía grant FQM-354.

The third author was supported as well by Agencia I+D+i grants PICT
2018-02501 and PICT 2019-00018.

All the authors were partially supported by Junta de Andalucía UMA18-FEDERJA-002.

\appendix

\section{Cotlar type inequalities}

Since we have not been able to find the following Cotlar type inequality
in the one-sided setting in the literature, we provide here a proof
for reader's convenience.
\begin{thm}
\label{thm:CotlarIneq}Let $A$ be a Young function. Let $T^{+}$
be a one-sided singular integral operator with associated kernel $K\in\mathcal{H}_{\bar{A}}$.
Then 
\[
(T^{+})^{*}f(x)\lesssim M_{\delta}^{+}(T^{+}f)(x)+M_{A}^{+}f(x)\qquad\delta\in(0,1).
\]
\end{thm}

Observe that if $T^{+}$ is a one-sided Calderón-Zygmund operator
its associated kernel is in particular a $\mathcal{H}_{\infty}$ kernel
and hence this result covers that case, with $M_{A}^{+}$ replaced
by $M^{+}$.
\begin{proof}[Proof of Theorem \ref{thm:CotlarIneq}]
Observe that it suffices to show that for every $\varepsilon>0$
\[
|T_{\varepsilon}^{+}f(0)|\lesssim M_{A}^{+}f(0)+M_{\delta}^{+}(T^{+}f)(0)
\]
where 
\[
T_{\varepsilon}^{+}f(0)=\int_{\varepsilon}^{\infty}K(0-y)f(y)dy.
\]
Observe that for $x>0$ we can write
\begin{align*}
T_{\varepsilon}^{+}f(0) & =T_{\varepsilon}^{+}f(0)-T^{+}f(x)+T^{+}f(x)\\
 & =T_{\varepsilon}^{+}f(0)-T^{+}f_{2}(x)-T^{+}f_{1}(x)+T^{+}f(x)
\end{align*}
where $f_{1}(x)=f(x)\chi_{(0,\varepsilon)}(x)$ and $f_{2}(x)=f(x)\chi_{(\varepsilon,\infty)}(x)$.
Then we have that
\[
\left|T_{\varepsilon}^{+}f(0)\right|\leq\left|T^{+}f_{1}(x)\right|+\left|T^{+}f(x)\right|+\left|T_{\varepsilon}^{+}f(0)-T^{+}f_{2}(x)\right|
\]
and $\delta$-averaging over $x$, if we call $h=\frac{\varepsilon}{2c_{A}}$
\begin{align*}
 & \left|T_{\varepsilon}^{+}f(0)\right|\\
 & \lesssim\left(\frac{1}{h}\int_{0}^{h}\left|T^{+}f_{1}(x)\right|^{\delta}dx\right)^{\frac{1}{\delta}}+\left(\frac{1}{h}\int_{0}^{h}\left|T^{+}f(x)\right|^{\delta}dx\right)^{\frac{1}{\delta}}+\left(\frac{1}{h}\int_{0}^{h}\left|T_{\varepsilon}^{+}f(0)-T^{+}f_{2}(x)\right|^{\delta}\right)^{\frac{1}{\delta}}\\
 & :=I+II+III.
\end{align*}
For $I$, we observe that any Hörmander condition implies that $K\in\mathcal{H}_{1}$
and hence by \cite{G}, we have that $T^{+}$ is of weak type $(1,1)$.
Then, by Kolmogorov inequality, 
\[
I\lesssim\frac{1}{h}\int_{\mathbb{R}}\left|f_{1}(x)\right|dx=\frac{2c_{A}}{\varepsilon}\int_{0}^{\varepsilon}\left|f(x)\right|dx\leq2c_{A}M^{+}f(0).
\]
For $II$
\[
II\leq M_{\delta}^{+}(T^{+}f)(0),
\]
and it remains to deal with $III.$ Observe that for $x\in(0,\frac{\varepsilon}{2c_{A}})$
\begin{align*}
T^{+}f_{2}(x) & =\lim_{\delta\rightarrow0}\int_{\delta+x}^{\infty}K(x-y)f_{2}(y)dy\\
 & =\int_{\varepsilon}^{\infty}K(x-y)f(y)dy=T_{\varepsilon}^{+}f(x).
\end{align*}
Bearing that in mind and since for every $x\in(0,\frac{\varepsilon}{2c_{A}})$
we have that $x<\frac{\varepsilon}{2c_{A}}$ or equivalently $xc_{A}<\frac{\varepsilon}{2}$,
we have that
\begin{align*}
 & \left|T_{\varepsilon}^{+}f(0)-T^{+}f_{2}(x)\right|\\
= & \left|T_{\varepsilon}^{+}f(0)-T_{\varepsilon}^{+}f(x)\right|\\
= & \left|\int_{\varepsilon}^{\infty}\left(K(-y)-K(x-y)\right)f(y)dy\right|\\
\leq & \int_{\varepsilon}^{\infty}\left|K(-y)-K(x-y)\right||f(y)|dy\\
\leq & \sum_{m=1}^{\infty}(2^{m}\frac{\varepsilon}{2})\frac{1}{(2^{m}\frac{\varepsilon}{2})}\int_{2^{m}\frac{\varepsilon}{2}<y\leq2^{m+1}\frac{\varepsilon}{2}}^{\infty}\left|K(-y)-K(x-y)\right||f(y)|dy\\
\lesssim & \sum_{m=1}^{\infty}(2^{m}\frac{\varepsilon}{2})\|(K(-\cdot)-K(x-\cdot))\chi_{2^{m}\frac{\varepsilon}{2}<y\leq2^{m+1}\frac{\varepsilon}{2}}\|_{\bar{A},(0,2^{m+1}\frac{\varepsilon}{2})}\|f\|_{A,(0,2^{m+1}\frac{\varepsilon}{2})}\\
\lesssim & C_{A}M_{A}^{+}f(0),
\end{align*}
and consequently
\[
III\lesssim C_{A}M_{A}^{+}f(0).
\]
This ends the proof.
\end{proof}

\end{document}